\documentclass[a4paper,11pt]{article}
\usepackage[utf8]{inputenc}
\usepackage{a4wide}

\usepackage{amsfonts}
\usepackage{amsmath}
\usepackage[mathscr]{eucal}
\usepackage{amssymb}
\usepackage{latexsym}
\usepackage{amsthm}
\usepackage{amscd}
\usepackage{epsf}
\usepackage{graphicx}
\usepackage{makeidx}
\usepackage{enumerate}
\usepackage{mathdots}
\usepackage{color}
\usepackage{hyperref}
\usepackage{url}
\usepackage[capitalize]{cleveref}

\providecommand{\keywords}[1]
{
  \noindent \small	
  \textbf{Keywords:} #1
}
\providecommand{\amscode}[1]
{
  \noindent \small	
  \textbf{AMS subject classifications:} #1
}

\newtheorem{theorem}{Theorem}[section]

\newtheorem{lemma}[theorem]{Lemma}

\newtheorem{definition}[theorem]{Definition}
\newtheorem{example}[theorem]{Example}
\newtheorem{remark}[theorem]{Remark}

\newcommand{\bsh}{\boldsymbol{h}}

\newcommand{\bsk}{\boldsymbol{k}}
\newcommand{\bsl}{\boldsymbol{l}}
\newcommand{\bsell}{\boldsymbol{\ell}}

\newcommand{\bsw}{\boldsymbol{w}}
\newcommand{\bsx}{\boldsymbol{x}}
\newcommand{\bsy}{\boldsymbol{y}}

\newcommand{\bszero}{\boldsymbol{0}}
\newcommand{\bsone}{\boldsymbol{1}}

\newcommand{\Ocal}{\mathcal{O}}
\newcommand{\Ucal}{\mathcal{U}}

\newcommand{\bN}{\mathbb{N}}

\newcommand{\bR}{\mathbb{R}}

\newcommand{\NN}{\mathbb{N}}

\newcommand{\Zb}{\mathbb{Z}_b}

\newcommand{\Fb}{\mathbb{F}_b}

\newcommand{\wal}{\mathrm{wal}}

\newcommand{\floor}[1]{\lfloor #1 \rfloor}

\DeclareMathOperator{\Ker}{Ker}
\DeclareMathOperator{\Image}{Image}

\DeclareMathOperator{\rank}{rank}
\DeclareMathOperator{\proj}{pr}

\begin{document}

\title{
Improved bounds on the gain coefficients for digital nets in prime power base
}

\author{
Takashi Goda\thanks{School of Engineering, University of Tokyo, 7-3-1 Hongo, Bunkyo-ku, Tokyo 113-8656, Japan ({\tt goda@frcer.t.u-tokyo.ac.jp})}, Kosuke Suzuki\thanks{Graduate School of Advanced Science and Engineering, Hiroshima University, 1-3-1 Kagamiyama, Higashi-Hiroshima City, Hiroshima 739-8526, Japan}
}

\date{\today}
\maketitle

\begin{abstract}
We study randomized quasi-Monte Carlo integration by scrambled nets. The scrambled net quadrature has long gained its popularity because it is an unbiased estimator of the true integral, allows for a practical error estimation, achieves a high order decay of the variance for smooth functions, and works even for $L^p$-functions with any $p\geq 1$. The variance of the scrambled net quadrature for $L^2$-functions can be evaluated through the set of the so-called \emph{gain coefficients.}

In this paper, based on the system of Walsh functions and the concept of dual nets, we provide improved upper bounds on the gain coefficients for digital nets in general prime power base. Our results explain the known bound by Owen (1997) for Faure sequences, the recently improved bound by Pan and Owen (2021) for digital nets in base 2 (including Sobol' sequences as a special case), and their finding that all the nonzero gain coefficients for digital nets in base 2 must be powers of two, all in a unified way.
\end{abstract}

\amscode{Primary 65C05, 65D30; Secondary 42C10, 43A40}
\section{Introduction}\label{sec:intro}
We study numerical integration of functions defined over the unit cube $[0,1)^s$ for a dimension $s$. For a Riemann integrable function $f: [0,1)^s\to \bR$, we denote the integral of $f$ by
\[ I(f) := \int_{[0,1)^s}f(\bsx)\, \mathrm{d}\bsx. \]
In this paper we consider an equally weighted quadrature to approximate $I(f)$:
\[ I(f; P_N) = \frac{1}{N}\sum_{n=0}^{N-1}f(\bsy_n),\]
for an $N$-element point set $P_N=\{\bsy_0,\dots,\bsy_{N-1}\}\subseteq [0,1)^s$. If we choose each $\bsy_n$ independently and randomly from the uniform distribution $\Ucal[0,1)^s$, this is nothing but the standard, unbiased Monte Carlo estimator and the variance of the estimator is given by $\sigma^2/N$ with $\sigma^2=I(f^2)-(I(f))^2$ for any function $f\in L^2([0,1)^s)$.

Randomized quasi-Monte Carlo (RQMC) methods are aimed at attaining an improved convergence rate of the estimator's variance instead of reducing the variance of functions $\sigma^2$ \cite{LE18}. The key ingredient of RQMC methods is that the point set $P_N$ is randomly distributed so that every point $\bsy_n$ follows $\Ucal[0,1)^s$ individually, but collectively the points $\bsy_0,\dots,\bsy_{N-1}$ are well equi-distributed over $[0,1)^s$. This way the resulting RQMC estimator not only becomes unbiased, but also allows for a practical error estimation as described in \cite[Section~2]{DKS13}. 

Among many others, the scrambled net quadrature due to Owen \cite{Ow95} has long gained its popularity as an RQMC method because of its additional desirable properties that \textbf{(1)} the variance of the estimator converges at a faster rate of $\Ocal(N^{-3+\varepsilon})$ with arbitrarily small $\varepsilon>0$ for smooth functions \cite{Ow97b,Ow98}, and that \textbf{(2)} it works even for $L_p$-functions with any $p\geq 1$ \cite{OR21,HR22}. The original variance analysis for the scrambled net quadrature was made in \cite{Ow97} by using the system of Haar functions, whereas, later in the book by Dick and Pillichshammer \cite{DP10}, it was shown that similar variance analysis is possible even by replacing it with the system of Walsh functions. In fact, this point is emphasized in the preface of \cite{DP10} as:
\begin{quote}
\textit{In the analysis of scrambled nets, no disadvantage seems to arise from replacing Haar functions with Walsh functions. The locality of Haar functions is offset by the locality of the Walsh-Dirichlet kernel. $\dots$ This makes Walsh functions more suitable for our endeavour than Haar functions.}
\end{quote}
We take the same standpoint and employ the system of Walsh functions in this paper.

Throughout this paper, we use the following notation. Let $\NN$ be the set of positive integers and $\NN_0 := \NN \cup \{0\}$. Given a prime power $b$, we denote the system of $b$-adic Walsh functions by $\{\wal_{\bsk}\}_{\bsk\in \NN_0^s}$. We refer to Subsection~\ref{subsec:walsh} for its definition. Noting that the Walsh system forms a complete orthonormal $L^2([0,1)^s)$ basis, any function $f\in L^2([0,1)^s)$ has its Walsh series
\[ f(\bsx)\sim \sum_{\bsk\in \NN_0^s}\hat{f}(\bsk)\wal_{\bsk}(\bsx),\]
where we denote by $f\sim g$ the equivalence relation for the $L^2([0,1)^s)$ space, i.e., $I((f-g)^2)=0$. Here, $\hat{f}(\bsk)$ denotes the $\bsk$-th Walsh coefficient of $f$:
\[ \hat{f}(\bsk):=\int_{[0,1)^s}f(\bsx)\overline{\wal_{\bsk}(\bsx)}\, \mathrm{d}\bsx.\]

For any underlying $N$-element point set $\{\bsx_0,\ldots,\bsx_{N-1}\}\subseteq [0,1)^s$ to which Owen's scrambling is applied, the scrambled net variance is given by
\[ \frac{1}{N}\sum_{\emptyset \neq u\subseteq \{1,\ldots,s\}}\sum_{\bsk\in \NN_0^{|u|}}\Gamma_{u,\bsk}\sigma_{u,\bsk}^2,\]
see \cite[Theorem~13.6]{DP10}, where $\{\Gamma_{u,\bsk}\}_{u,\bsk}$ denotes the set of the so-called \emph{gain coefficients}, the non-negative values depending on the underlying points $\bsx_0,\ldots,\bsx_{N-1}$ but not on a function $f$, and we write 
\[ \sigma_{u,\bsk}^2 = \sum_{\substack{ b^{k_j}\leq \ell_j<b^{k_j+1}\\ j\in u}}|\hat{f}(\bsell_u,\bszero)|^2, \]
where we denote by $(\bsell_u,\bszero)$ the vector $\bsh\in \NN_0^s$ such that $h_j=\ell_j$ if $j\in u$ and $h_j=0$ otherwise. This is why the set of the gain coefficients plays an important role in evaluating the effectiveness of the scrambled net quadrature and is the subject of this paper. 

In passing, it follows from  Parseval's identity that the variance of the standard Monte Carlo estimator can be expressed by
\[ \frac{\sigma^2}{N} = \frac{1}{N}\sum_{\emptyset \neq u\subseteq \{1,\ldots,s\}}\sum_{\bsk\in \NN_0^{|u|}}\sigma_{u,\bsk}^2.\]
As we can see, the scrambled net variance is bounded above by
\[ \frac{\Gamma \sigma^2}{N}\quad \text{with}\quad \Gamma := \max_{\emptyset \neq u\subseteq \{1,\ldots,s\}}\max_{\bsk\in \NN_0^{|u|}}\Gamma_{u,\bsk},\]
so that the scrambled net quadrature with $N$ points performs no wore than the standard Monte Carlo estimator with $\lfloor N/\Gamma\rfloor$ points for any function $f\in L^2([0,1)^s)$ in term of variance. For smooth functions, we can exploit the decay of $\sigma_{u,\bsk}^2$ to show that the scrambled net variance with some good underlying points decays much faster than $\Ocal(1/N)$ \cite{Ow97,Ow97b,Ow98,YM99,DP10}.

For special types of the underlying point sets, called $(t,m,s)$-nets and digital nets, there are some upper bounds on the gain coefficients known in the literature. Regarding the definitions of these special point sets, we defer to the next section. It was first shown in \cite[Theorem~3]{Ow97} that $(0,m,s)$-nets in base $b$ satisfy a uniform bound $\Gamma_{u,\bsk}\leq (b/(b-1))^{s-1}$ for all $u$ and $\bsk$. The first $b^m$ points of Faure's digital sequences \cite{Fa82} is a $(0,m,s)$-net in which the base $b$ needs to be a prime power larger than or equal to the dimension $s$, so that we obtain $\Gamma_{u,\bsk}\leq (b/(b-1))^{b-1}\leq e$. For more general $(t,m,s)$-nets in base $b$, \cite[Lemmas~2--4]{Ow98} proved a bound
\[ \Gamma_{u,\bsk}\leq \begin{cases} 0 & \text{if $|\bsk|\leq m-t-|u|,$} \\ \displaystyle b^t\left(\frac{b+1}{b-1}\right)^{|u|} & \text{if $m-t-|u|<|\bsk|<m-t ,$}\\ \displaystyle b^t\frac{b^{|u|}+(b-2)^{|u|}}{2(b-1)^{|u|}} & \text{if $|\bsk|\geq m-t,$}\end{cases}\]
where we write $|\bsk|=\sum_{j\in u}k_j$. For the case $|\bsk|\geq m-t$, this result was improved in \cite[Proposition~5.1]{NP01} by analyzing the microstructure of $(t,m,s)$-nets. In \cite[Lemma~6 and Theorem~7]{YH02}, Yue and Hickernell considered digital nets in prime power base $b$ and gave exact formulas of the gain coefficients in terms of generating matrices. In a recent work by Pan and Owen \cite{PO21}, essentially the same formula was derived in Theorem~1 for digital $(t,m,s)$-nets in base 2, and then a bound $\Gamma_{u,\bsk}\leq 2^{t+|u|-1}$ for any $u$ and $\bsk$ was proven in Corollary~3. Moreover, an interesting finding in \cite{PO21} is that $\Gamma_{u,\bsk}$ for digital nets in base 2 must be 0 or powers of two, as shown in \cite[Theorem~2]{PO21}.

The aim of this paper is to generalize many of the results for digital nets in base 2 by Pan and Owen \cite{PO21} and to provide improved upper bounds on the gain coefficients for digital nets in general prime power base $b$. In fact, our obtained bounds can be regarded as a straightforward consequence of the results shown in \cite[Chapter~13]{DP10} where the scrambled net variance is analyzed based on the system of Walsh functions. However, by investigating the gain coefficients in more detail through the lens of duality theory for digital nets originally developed by Niederreiter and Pirsic \cite{NP01b} and Skriganov \cite{Sk01}, see also \cite[Chapter~7]{DP10}, we can provide a unified result, explaining how special the recent finding that the non-zero $\Gamma_{u,\bsk}$ for digital nets in base 2 must be powers of two is. 

The rest of this paper is organized as follows. In the next section, we introduce necessary definitions such as $(t,m,s)$-nets, digital nets, dual nets, and Walsh functions. In Section~\ref{sec:bound}, we show improved upper bounds on the gain coefficients via an argument relying on the system of Walsh functions and the concept of dual nets. Although our main result in Theorem~\ref{thm:gain_bound} can be seen as a straightforward extension of the results in \cite[Chapter~13]{DP10}, we give yet another proof of the key Lemma~\ref{lem:bound_A}, by which it enables us to understand that the result in \cite{PO21} that the non-zero gain coefficients must be powers of two is special for digital nets in base 2. That is, for a prime power base $b>2$, the non-zero gain coefficients are not necessarily powers of $b$. Finally in Section~\ref{sec:maximal}, we discuss reduced upper bounds on the maximal gain coefficient and its related quantities, which apply to digital nets in general prime power base.

\section{Preliminaries}

\paragraph{Notation.}
Again, let $\NN$ be the set of positive integers and $\NN_0 := \NN \cup \{0\}$. Let $b$ be a prime power, $\Zb = \{0,1,\dots, b-1\}$ be the residue ring modulo $b$, and $\Fb$ be the $b$-element finite field. Let $\varphi: \Zb\to \Fb$ be a bijection with $\varphi(0)=0$. For a vector $\bsk = (k_1,\dots,k_s) \in \NN_0^s$, define $|\bsk| := \sum_{j=1}^s k_j$. For two vectors $\bsk,\bsk'$ with the same dimensionality, the notation $\bsk\ge \bsk'$ denotes the componentwise inequality. Moreover we write $1{:}s = \{1,2,\dots,s\}$.

\subsection{\texorpdfstring{$(t,m,s)$}{(t,m,s)}-nets and digital nets}\label{subsec:digital_net}

\begin{definition}[$(t,m,s)$-nets]
For $m\in \NN_0$, let $P \subseteq [0,1)^s$ be a set of $b^m$ points. We call $P$ a $(t,m,s)$-net in base $b$ if there exists a $t\in \NN_0$ such that every interval of the form
\[
\prod_{j=1}^s \left[\frac{a_j}{b^{c_j}}, \frac{a_j+1}{b^{c_j}}\right) \quad \text{with $c_j \in \NN_0$, $0 \le a_j < b^{c_j}$ and $\sum_{j=1}^s c_j = m-t$}
\]
contains exactly $b^t$ points of $P$. A $(t,m,s)$-net in base $b$ is called a strict $(t,m,s)$-net in base $b$, if it is not a $(t-1,m,s)$-net in base $b$.
\end{definition}

By definition, $(t,m,s)$-nets with smaller $t$-value are more equi-distributed, as they satisfy the same condition with respect to finer partitions of $[0,1)^s$.

\begin{definition}[digital nets]\label{def:digital_net}
For $m,n\in \NN$, let $C_1,\ldots,C_s$ be $n\times m$ matrices over $\Fb$. For an integer $0\leq h<b^m$ with $b$-adic expansion $h = \eta_0 + \eta_1 b+\cdots + \eta_{m-1}b^{m-1}$, define the point $\bsx_h=(x_{h,1},\ldots,x_{h,s})\in [0,1)^s$ by
\[ x_{h,j} = \frac{\varphi^{-1}(\xi_{1,h,j})}{b}+\frac{\varphi^{-1}(\xi_{2,h,j})}{b^2}+\cdots + \frac{\varphi^{-1}(\xi_{n,h,j})}{b^n}, \]
where
\[ (\xi_{1,h,j},\xi_{2,h,j},\ldots, \xi_{n,h,j})^\top = C_j \cdot (\varphi(\eta_0),\varphi(\eta_1),\ldots,\varphi(\eta_{m-1}))^{\top}. \]
We call $P=\{ \bsx_h \mid 0\leq h<b^m\}\subseteq [0,1)^s$ a digital net in base $b$. If $P$ is a (strict) $(t,m,s)$-net in base $b$, we call $P$ a (strict) digital $(t,m,s)$-net in base $b$.
\end{definition}

The strict $t$-value of digital nets is computable from their generating matrices $C_1,\ldots,C_s$.
We first define matrices which are collections of row vectors of $C_1,\ldots,C_s$.

\begin{definition}
For $m,n\in \NN$ with $n\geq m$, let $C_1,\ldots,C_s$ be $n\times m$ matrices over $\Fb$. Let $u = \{r_1, \dots, r_{|u|}\} \subseteq 1{:}s$ and $\bsk \in \NN_0^{|u|}$.
We define $C_{u,\bsk} \in \Fb^{|\bsk|\times m}$
by
\[
C_{u,\bsk} :=
\begin{pmatrix}
\text{the first $k_1$ row vectors of $C_{r_1}$}\\
\text{the first $k_2$ row vectors of $C_{r_2}$}\\
\vdots\\
\text{the first $k_{|u|}$ row vectors of $C_{r_{|u|}}$}\\
\end{pmatrix}.
\]
\end{definition}
\noindent
Then it holds that the strict $t$-value of a digital net is equal to
\begin{equation}\label{eq:tvalue_rank}
t = m+1 - \min_{\emptyset \neq u\subseteq 1{:}s}\min_{\bsk \in \NN_0^{|u|}} \left\{|\bsk| \mid \text{$C_{u,\bsk}$ is not full row rank}\right\}.
\end{equation}
We refer to \cite[Theorem~4.52]{DP10} for the proof.

The concept of the dual net for a digital net $P$, defined below, plays a crucial role in our subsequent analysis of the gain coefficients. 
\begin{definition}[Dual nets]\label{def:dual}
For $m,n\in \NN$, let $P$ be a digital net over $\Fb$ with generating matrices $C_1,\ldots,C_s\in \Fb^{n\times m}$. The dual net of $P$, denoted by $P^{\perp}$, is defined by
\[ P^{\perp} := \left\{ \bsk=(k_1,\ldots,k_s)\in \NN_0^s \mid C_1^{\top}\nu_n(k_1) + \cdots + C_s^{\top}\nu_n(k_s)=\bszero \in \Fb^m \right\}, \]
with
\[ \nu_n(k)=(\varphi(\kappa_0),\ldots,\varphi(\kappa_{n-1}))^{\top}\in \Fb^n \]
for $k\in \NN_0$ with $b$-adic expansion $k=\kappa_0+\kappa_1 b+\cdots$, where all but a finite number of $\kappa_i$ are 0.
\end{definition}

\subsection{Walsh functions}\label{subsec:walsh}
In what follows, let $\omega_b$ denote the primitive $b$-th root of unity $\exp(2\pi \sqrt{-1}/b)$. The univariate Walsh functions are defined as follows.
\begin{definition}
Let $k\in \NN_0$ with $b$-adic expansion $k=\kappa_0+\kappa_1 b+\cdots$, where all but a finite number of $\kappa_i$ are 0. The $k$-th $b$-adic Walsh function $\wal_k\colon [0,1)\to \{1,\omega_b,\ldots,\omega_b^{b-1}\}$ is defined by
\[ \wal_k (x) := \omega_b^{\varphi^{-1}(\varphi(\kappa_0)\varphi(\xi_1))+\varphi^{-1}(\varphi(\kappa_1)\varphi(\xi_2))+\cdots},\]
where the $b$-adic expansion of $x\in [0,1)$ is denoted by $x=\xi_1 b^{-1}+\xi_2 b^{-2}+\cdots$, which is unique in the sense that infinitely many of the digits $\xi_i$ are different from $b-1$.
\end{definition}
\noindent For the multivariate case, the Walsh functions are defined as follows.
\begin{definition}
For $s\geq 1$ and $\bsk=(k_1,\ldots,k_s)\in \NN_0^s$, the $\bsk$-th $b$-adic Walsh function $\wal_{\bsk}\colon [0,1)^s\to \{1,\omega_b,\ldots,\omega_b^{b-1}\}$ is defined by
\[ \wal_{\bsk} (\bsx) := \prod_{j=1}^{s}\wal_{k_j}(x_j). \]
\end{definition}

Here we compute the sum of the first $b^k$ Walsh functions, called the Walsh-Dirichlet kernel \cite[Appendix~A]{DP10}, for later use.
For $s=1$, by denoting the $b$-adic expansion of $l\in \{0,\ldots,b^k-1\}$ by $l=\ell_0+\ell_1b+\cdots+\ell_{k-1}b^{k-1}$ and using the fact that $\varphi$ is a bijection with $\varphi(0)=0$, we have
\[ \sum_{l=0}^{b^k-1} \wal_l(x)
= \prod_{i=0}^{k-1} \sum_{\ell_i=0}^{b-1} \omega_b^{\varphi^{-1}(\varphi(\ell_i) \varphi(\xi_{i+1}))}
= \prod_{i=0}^{k-1} b \chi(\xi_{i+1} = 0)
= b^{k} \chi(\floor{b^kx}=0),\]
where $\chi(A)$ denotes the indicator function of an event $A$. For $s>1$, this equality can be simply generalized as
\begin{equation}\label{eq:sum_in_V}
\sum_{\substack{\bsl \in \NN_0^s\\ l_j<b^{k_j},\, j\in u\\ l_j=0,\, j\not\in u}} \wal_{\bsl}(\bsx)
= \prod_{j \in u} \sum_{l_j = 0}^{b^{k_j}-1} \wal_{l_j}(x_j)
= \prod_{j \in u} b^{k_j}\chi(\floor{b^{k_j} x_j}=0)
= b^{|\bsk|}\prod_{j \in u} \chi(\floor{b^{k_j} x_j}=0),
\end{equation}
for any $u \subseteq 1{:}s$ and $\bsk \in \NN_0^{|u|}$.

The following character property is also important in the subsequent analysis. We refer to \cite[Lemma~4.75]{DP10} for the proof.
\begin{lemma}\label{lem:dual-walsh}
Let $P\subseteq [0,1)^s$ be a digital net in base $b$. For any $\bsk\in \NN_0^s$ we have
\[ \sum_{\bsx\in P}\wal_{\bsk}(\bsx) = \begin{cases} |P| & \text{if $\bsk\in P^{\perp}$,} \\ 0 & \text{otherwise,}\end{cases} \]
where $P^{\perp}$ denotes the dual net of $P$ as defined in Definition~\ref{def:dual}.
\end{lemma}

\section{Bounds on gain coefficients via a dual net approach}\label{sec:bound}
\subsection{Improved upper bounds}
Building upon \cite[Chapter~13]{DP10}, we give improved upper bounds on the gain coefficients. First let us give a formal definition of the gain coefficients introduced by Owen \cite{Ow97}.

\begin{definition}\label{def:gain}
Let $P_N=\{\bsx_0,\ldots,\bsx_{N-1}\}\subseteq [0,1)^s$ be an $N$-element point set. For a non-empty $u \subseteq 1{:}s$ and $\bsk=(k_j)_{j\in u}\in \bN_0^{|u|}$, we define the gain coefficient of $P_N$ by
\[
\Gamma_{u,\bsk}
:= \frac{1}{N} \sum_{i,i'=0}^{N-1}\prod_{j \in u}\frac{b \chi(\floor{b^{k_j+1}x_{i,j}}=\floor{b^{k_j+1}x_{i',j}})
- \chi(\floor{b^{k_j}x_{i,j}}=\floor{b^{k_j}x_{i',j}})}{b-1}.
\]
\end{definition}
\noindent
We note that this definition applies to any finite point set, and that our $\Gamma_{u,\bsk}$ is exactly the same with
the quantity $G_{\bsk + \bsone_u}$ used in \cite[Section~13.3]{DP10}, where $\bsk+\bsone_u=(k_j+1)_{j\in u}$.

As preparation, let us introduce some subsets of $\NN_0^s$.
\begin{definition}\label{def:VandA}
For $u \subseteq 1{:}s$ and $\bsk \in \NN_0^{|u|}$, we define the sets $V(u,\bsk)$ and $A(u,\bsk)$ by
\begin{align*}
V(u,\bsk) &:= \{\bsl \in \NN_0^s \mid \text{$l_j < b^{k_j}$ if $j \in u$ and $l_j = 0$ otherwise} \}\quad \text{and}\\
A(u,\bsk) &:= \{\bsl \in \NN_0^s \mid \text{$b^{k_j - 1} \le l_j < b^{k_j}$ if $j \in u$ and $l_j = 0$ otherwise} \},
\end{align*}
respectively.
\end{definition}
\noindent Note that the equality \eqref{eq:sum_in_V} on the Walsh-Dirichlet kernel is equivalent to
\[ \sum_{\bsl \in V(u,\bsk)} \wal_{\bsl}(\bsx)
= b^{|\bsk|}\prod_{j \in u} \chi(\floor{b^{k_j} x_j}=0).\]
Moreover, it follows from the inclusion-exclusion principle that
\begin{equation}\label{eq:in_ex}
|A(u,\bsk+\bsone_u)| = 
\sum_{v \subseteq u} (-1)^{|u|-|v|} |V(u,\bsk+\bsone_v)|,
\end{equation}
holds for any $u \subseteq 1{:}s$, where $\bsk+\bsone_v$ denotes the vector $\bsell\in \NN_0^{|u|}$ such that $\ell_j=k_j+1$ if $j\in v$ and $\ell_j=k_j$ if $j\in u\setminus v$.

We show below that the gain coefficients can be expressed in a simplified form when the underlying points $\bsx_0,\ldots,\bsx_{N-1}$ are a digital net in base $b$. 
Although essentially the same result can be found in the proof of \cite[Corollary~13.7]{DP10}, we give a proof for the sake of completeness.

\begin{lemma}\label{lem:gain_via_counting}
Let $P_N$ be a digital net in base $b$ with $N=b^m$ points. Then, for any non-empty $u \subseteq 1{:}s$ and $\bsk=(k_j)_{j\in u}\in \bN_0^{|u|}$, it holds that
\[
\Gamma_{u,\bsk}
= \frac{b^{m-|\bsk|}}{(b-1)^{|u|}} |A(u,\bsk+\bsone_u) \cap P_N^\perp|.
\]
\end{lemma}

\begin{proof}
In this proof, we denote the digit-wise $b$-adic subtraction by $\ominus$. That is, for $x,x'\in [0,1)$ with $x=\xi_1b^{-1}+\xi_2b^{-2}+\cdots$ and $x'=\xi'_1b^{-1}+\xi'_2b^{-2}+\cdots$, we write
\[ x\ominus x' = \frac{\eta_1}{b}+\frac{\eta_2}{b^2}+\cdots , \quad \text{with $\eta_i=\varphi^{-1}(\varphi(\xi_i)-\varphi(\xi'_i))$,}\]
provided that infinitely many $\eta_i$ are different from $b-1$.
For the case of vectors in $[0,1)^s$, we apply $\ominus$ componentwise. We also apply $\ominus$ to vectors in $\NN_0^s$ based on their $b$-adic (finite) expansions. Then we see from Definition~\ref{def:digital_net} that $\bsx_i \ominus \bsx_{i'} = \bsx_{i \ominus i'}$ holds for any $i,i'\in \{0,\ldots,b^m-1\}$.
Thus, together with the equality on the Walsh-Dirichlet kernel and Lemma~\ref{lem:dual-walsh}, it follows that
\begin{align*}
\frac{1}{N} \sum_{i,i'=0}^{N-1} \prod_{j \in u} \chi(\floor{b^{k_j}x_{i,j}}=\floor{b^{k_j}x_{i',j}})
&= \sum_{i=0}^{N-1} \prod_{j \in u} \chi(\floor{b^{k_j}x_{i,j}}=0) \\
&= b^{-|\bsk|}\sum_{\bsx\in P_N} \sum_{\bsl \in V(u,\bsk)} \wal_{\bsl}(\bsx) \\
&= b^{-|\bsk|} \sum_{\bsl \in V(u,\bsk)} \sum_{\bsx\in P_N} \wal_{\bsl}(\bsx) \\
&= b^{m-|\bsk|} |V(u,\bsk) \cap P_N^\perp|.
\end{align*}

Using this result and the inclusion-exclusion principle \eqref{eq:in_ex}, we have
\begin{align*}
\Gamma_{u,\bsk}
&= \frac{1}{N} \sum_{i,i'=0}^{N-1}\prod_{j \in u}\frac{b \chi(\floor{b^{k_j+1}x_{i,j}}=0) - \chi(\floor{b^{k_j}x_{i,j}}=0)}{b-1}\\
&= \frac{1}{(b-1)^{|u|}} \sum_{i=0}^{N-1}
\sum_{v \subseteq u} b^{|v|}(-1)^{|u|-|v|}
\prod_{j \in v} \chi(\floor{b^{k_j+1}x_{i,j}}=0) \prod_{j \in u \setminus v} \chi(\floor{b^{k_j}x_{i,j}}=0)\\
&= \frac{1}{(b-1)^{|u|}}
\sum_{v \subseteq u} b^{|v|}(-1)^{|u|-|v|} \cdot b^{m-|\bsk+\bsone_v|} |V(u,\bsk+\bsone_v) \cap P_N^\perp| \\
&= \frac{b^{m-|\bsk|}}{(b-1)^{|u|}}
\sum_{v \subseteq u} (-1)^{|u|-|v|} |V(u,\bsk+\bsone_v) \cap P_N^\perp| \\
&= \frac{b^{m-|\bsk|}}{(b-1)^{|u|}}|A(u,\bsk+\bsone_u) \cap P_N^\perp|.
\end{align*}
This completes the proof.
\end{proof}

A bound on $|A(u,\bsk+\bsone_u) \cap P_N^\perp|$ was given in \cite[Lemma~13.8]{DP10} with the proof based on counting the number of solutions of a linear equation over $\Fb$. In fact, the argument made in the proof can be easily strengthened to obtain the following bound.
Note that we will give a different proof in Section~\ref{sec:proof}.

\begin{lemma}\label{lem:bound_A}
Let $P_N$ be a digital $(t,m,s)$-net in base $b$ with $N=b^m$ points. Then, for any non-empty $u \subseteq 1{:}s$ and $\bsk=(k_j)_{j\in u}\in \bN_0^{|u|}$, we have
\[
|A(u,\bsk + \bsone_u) \cap P_N^\perp| \le
\begin{cases}
0 & \text{if $|\bsk| \le m-t-|u|$},\\
(b-1)^{|u|+|\bsk|-(m-t)} & \text{if $m-t-|u| < |\bsk| \le m-t$},\\
(b-1)^{|u|}b^{|\bsk|-(m-t)} & \text{if $|\bsk| > m-t$}.
\end{cases}
\]
\end{lemma}

By using Lemmas \ref{lem:gain_via_counting} and \ref{lem:bound_A}, we obtain the following bound on the gain coefficients.
\begin{theorem}\label{thm:gain_bound}
Let $P_N$ be a digital $(t,m,s)$-net in base $b$ with $N=b^m$ points.
Then, for any non-empty $u \subseteq 1{:}s$, we have
\[
\Gamma_{u,\bsk} \le
\begin{cases}
0 & \text{if $|\bsk| \le m-t-|u|$},\\
b^{m-|\bsk|}(b-1)^{|\bsk|-m+t} \le \displaystyle \frac{b^{t+|u|-1}}{(b-1)^{|u|-1}} & \text{if $m-t-|u| < |\bsk| \le m-t$},\\
b^{t} & \text{if $|\bsk| > m-t$}.\\
\end{cases}
\]
In particular we have $\Gamma \le b^{t+s-1}/(b-1)^{s-1}$.
\end{theorem}

\begin{remark}
Our result covers \cite[Corollary~3]{PO21} which only applies to digital nets in base 2. Furthermore, if $P_N$ is a digital $(0,m,s)$-net in base $b\geq s$, such as the first $b^m$ points of Faure sequence, we get $\Gamma_{u,\bsk}\leq  (b/(b-1))^{s-1}$, which was obtained previously in \cite{Ow97}.
\end{remark}

\subsection{Yet another proof of Lemma~\ref{lem:bound_A}} \label{sec:proof}

From now on, with a different approach from the proof of \cite[Lemma~13.8]{DP10}, we give upper bounds on $|A(u,\bsk+\bsone_u) \cap P^\perp|$ and other related quantities for $P$ being a digital net in prime power base $b$. In \cite[Remark~1]{PO21}, Pan and Owen stated that their approach does not work for general base, whereas our dual net approach enables to generalize some of their results to prime power base.

In what follows, the sets $V(u,\bsk)$ and $A(u,\bsk)$ as given in Definition~\ref{def:VandA} are both understood to be endowed with digit-wise addition and multiplication over $\Fb$.
To proceed our argument, let us define a projection
$\proj_{u,\bsk+\bsone_u} \colon V(u,\bsk+\bsone_u) \to \Fb^{|u|}$ by
\[
\proj_{u,\bsk+\bsone_u} (l_1,\dots,l_s) = (l_{j,k_j+1})_{j \in u}
\qquad \text{where $l_j = \sum_{i=1}^{k_j+1} l_{j,i}b^{i-1}$}
\]
and also define its restriction on $V(u,\bsk+\bsone_u) \cap P^\perp$ by
\[
\tilde{\proj}_{u,\bsk+\bsone_u} := \proj_{u,\bsk+\bsone_u}|_{V(u,\bsk+\bsone_u) \cap P^\perp}.
\]
Note that both $\proj_{u,\bsk+\bsone_u}$ and $\tilde{\proj}_{u,\bsk+\bsone_u}$ are $\Fb$-linear maps. Further, we define
\[
Q(u,\bsk+\bsone_u,P) := \Image (\tilde{\proj}_{u,\bsk+\bsone_u}) \cap (\Fb \setminus \{0\})^{|u|}.
\]
Our key observation is the following lemma.
\begin{lemma}\label{lem:decompose_A}
For any $u \subseteq 1{:}s$ and $\bsk \in \NN_0^{|u|}$, we have
\[
|A(u,\bsk + \bsone_u) \cap P^\perp|
= |V(u,\bsk) \cap P^\perp|
\times |Q(u,\bsk+\bsone_u,P)|.
\]
\end{lemma}

\begin{proof}
It holds that $\tilde{\proj}_{u,\bsk+\bsone_u}$ is a 
$|V(u,\bsk) \cap P^\perp|:1$ map since we have
\[
\Ker \tilde{\proj}_{u,\bsk+\bsone_u} = V(u,\bsk) \cap P^\perp.
\]
On the other hand, we have
\[
A(u,\bsk + \bsone_u) \cap P^\perp = \proj_{u,\bsk+\bsone_u}^{-1}(Q(u,\bsk + \bsone_u,P)).
\]
The lemma follows from these two results.
\end{proof}

Let us give a bound on $|V(u,\bsk)\cap P^\perp|$.
Since $V(u,\bsk)\cap P^\perp$ can be seen as the kernel of $C_{u,\bsk}^\top$, it follows that
\begin{equation}\label{eq:V_rank}
\dim (V(u,\bsk) \cap P^\perp) = \dim \Ker C_{u,\bsk}^\top = |\bsk| - \rank (C_{u,\bsk}).
\end{equation}
Furthermore, we have a bound on $\rank (C_{u,\bsk})$ from \eqref{eq:tvalue_rank},
i.e.,
$\rank (C_{u,\bsk}) = |\bsk|$ if $|\bsk| \le m-t$ and 
$\rank (C_{u,\bsk}) \ge m-t$ if $|\bsk| > m-t$.
These results imply the following bound.

\begin{lemma}\label{lem:bound_V}
Let $P$ be a digital $(t,m,s)$-net in base $b$.
Then, for any $u \subseteq 1{:}s$ and $\bsk \in \NN_0^{|u|}$, we have 
\[
|V(u,\bsk) \cap P^\perp|
\le 
\begin{cases}
1 & \text{if $|\bsk| \le m-t$,}\\
b^{|\bsk|-m+t} & \text{otherwise}.
\end{cases}
\]
\end{lemma}

Let us move on to bounding $|Q(u,\bsk+\bsone_u,P)|$.
We need the following lemma.
\begin{lemma}\label{lem:bound_nonzero}
Let $W \subseteq \Fb^n$ be a subspace
and $d := \dim W$. Then we have
\[
|W \cap (\Fb\setminus\{0\})^n| \le
\begin{cases}
0 & \text{if $d = 0$,}\\
(b-1)^{d} & \text{otherwise}.
\end{cases}
\]
\end{lemma}

\begin{proof}
As the result for $d = 0$ is trivial, let us consider the case $d \ge 1$.
Let $(\bsw_1, \dots, \bsw_d)$ be a reduced column echelon form of $W$.
Let us assume $a_j \in \Fb$ and
$
\sum_{j=1}^d a_j \bsw_j \in (\Fb\setminus\{0\})^n.
$
Then, by the definition of a reduced column echelon form,
each $a_j$ must be non-zero.
Thus the number of $a_j$'s does not exceed $(b-1)^d$.
\end{proof}

To apply the above lemma to show a bound on $|Q(u,\bsk+\bsone_u,P)|$,
it is required to know $\dim \Image (\tilde{\proj}_{u,\bsk+\bsone_u})$.
In fact, we have
\begin{align}
\dim \Image (\tilde{\proj}_{u,\bsk+\bsone_u})
& = 
\dim (V(u,\bsk+\bsone_u) \cap P^\perp) - \dim \Ker \tilde{\proj}_{\bsk+\bsone_u} \notag\\
& = \dim (V(u,\bsk+\bsone_u) \cap P^\perp) - \dim (V(u,\bsk) \cap P^\perp) \notag\\
& = (|\bsk| + |u| - \rank(C_{u,\bsk+\bsone_u})) - (|\bsk| - \rank(C_{u,\bsk})) \notag\\
& = |u| - \rank(C_{u,\bsk+\bsone_u}) + \rank(C_{u,\bsk}), \label{eq:im_rank}
\end{align}
where we used \eqref{eq:V_rank} in the third equality.
Thus it follows from Lemma~\ref{lem:bound_nonzero} and \eqref{eq:im_rank} that
\begin{equation}\label{eq:Q_rank}
|Q(u,\bsk+\bsone_u,P)| 
\le
\begin{cases}
0 & \text{if $\rank(C_{u,\bsk+\bsone_u})-\rank(C_{u,\bsk}) = |u|$,}\\
(b-1)^{|u|-\rank(C_{u,\bsk+\bsone_u})+\rank(C_{u,\bsk})} & \text{otherwise.}
\end{cases}
\end{equation}
By using \eqref{eq:tvalue_rank} again, we obtain the following bound.

\begin{lemma}\label{lem:bound_Q}
Let $P$ be a digital $(t,m,s)$-net in base $b$.
Then, for any $u \subseteq 1{:}s$ and $\bsk \in \NN_0^{|u|}$, we have 
\[
|Q(u,\bsk+\bsone_u,P)| 
\le 
\begin{cases}
0 & \text{if $|\bsk| \le m-t-|u|$},\\
(b-1)^{|u|+|\bsk|-(m-t)} & \text{if $m-t-|u| < |\bsk| \le m-t$},\\
(b-1)^{|u|} & \text{if $|\bsk| > m-t$}.\\
\end{cases}
\]
\end{lemma}

We find that Lemma~\ref{lem:bound_A} is now apparent from
Lemmas~\ref{lem:decompose_A}, \ref{lem:bound_V} and \ref{lem:bound_Q}.

\begin{remark}
Let us focus on the case $b=2$. By combining our Lemmas~\ref{lem:gain_via_counting}, \ref{lem:decompose_A} and \eqref{eq:V_rank}, we have
\[ \Gamma_{u,\bsk}=2^{m-\rank (C_{u,\bsk})}|Q(u,\bsk+\bsone_u,P)|.\]
Here, as we have
$(\mathbb{F}_2 \setminus \{0\})^{|u|} = \{(1,\dots,1)\}$
and thus $|Q(u,\bsk+\bsone_u,P)|$ must be either $0$ or $1$, $\Gamma_{u,\bsk}$ must be either $0$ or $2^{m-\rank (C_{u,\bsk})}$. Together with \eqref{eq:Q_rank}, our dual net approach recovers \cite[Theorem~2 and Corollary~4]{PO21} exactly. Otherwise if $b>2$, non-zero $\Gamma_{u,\bsk}$'s are not necessarily powers of $b$. Example~\ref{example} supports our claim.
\end{remark}

\section{Maximal gain coefficients}\label{sec:maximal}

We conclude this paper with giving some results on the maximal gain coefficients. Recall that, as introduced in Section~\ref{sec:intro}, the original maximal gain coefficient is defined by
\[ \Gamma := \max_{\emptyset \neq u\subseteq \{1,\ldots,s\}}\max_{\bsk\in \NN_0^{|u|}}\Gamma_{u,\bsk}. \]
Moreover we can define the maximal gain coefficients for any non-empty subset $u\subseteq 1{:}s$ as
\begin{align*}
    \Gamma_u & := \max_{\bsk\in \NN_0^{|u|}} \Gamma_{u,\bsk},\quad \text{and}\\
    \Gamma^*_u & := \max_{\emptyset \neq v \subseteq u} \max_{\bsk\in \NN_0^{|v|}} \Gamma_{v,\bsk}.
\end{align*}
By definition, we have 
\[ \Gamma=\Gamma_{1{:}s}^*=\max_{\emptyset \neq u\subseteq 1{:}s}\Gamma_{u}. \]

\begin{remark}\label{rem:full_rank}
Here we note that, if $C_{u, \bsk+\bsone_u}$ is full row rank, then $C_{u, \bsk}$ is trivially full row rank, so that we have $|Q(u,\bsk+\bsone_u,P)|=0$ by \eqref{eq:Q_rank}. It follows from Lemmas~\ref{lem:decompose_A} and \ref{lem:gain_via_counting} that $\Gamma_{u,\bsk}=0$. 
\end{remark}

First let us introduce an upper bound on the gain coefficient
\begin{equation}
B(u,\bsk) := b^{m - \rank (C_{u,\bsk})} (b-1)^{\rank(C_{u,\bsk}) - \rank(C_{u,\bsk + \bsone_u})},
\end{equation}
for a non-empty $u \subseteq 1{:}s$ and $\bsk \in \NN_0^{|u|}$. In fact, it follows from Lemmas~\ref{lem:gain_via_counting} and \ref{lem:decompose_A}, \eqref{eq:V_rank} and \eqref{eq:Q_rank}, that
\begin{align*}
    \Gamma_{u,\bsk} & = \frac{b^{m-|\bsk|}}{(b-1)^{|u|}} |V(u,\bsk) \cap P^\perp| \times |Q(u,\bsk+\bsone_u,P)|\\
    & \le \frac{b^{m-|\bsk|}}{(b-1)^{|u|}} b^{|\bsk| - \rank (C_{u,\bsk})} \times (b-1)^{|u|-\rank(C_{u,\bsk+\bsone_u})+\rank(C_{u,\bsk})}\\
    & = b^{m - \rank (C_{u,\bsk})} (b-1)^{\rank(C_{u,\bsk}) - \rank(C_{u,\bsk + \bsone_u})}=B(u,\bsk).
\end{align*}

\begin{lemma}\label{lem:reduce}
Regarding the quantity $B(u,\bsk)$, the following properties hold true for any non-empty $u\subseteq 1{:}s$.
\begin{enumerate}
    \item For any $\bsk, \bsk'\in \NN_0^{|u|}$ with $\bsk\geq \bsk'$, we have $B(u,\bsk) \le B(u,\bsk')$.
    \item For any $v\subseteq u$ and $\bsk'\in \NN_0^{|v|}$, let $\bsk = (\bsk', \bszero_{u \setminus v}) \in \NN_0^{|u|}$. Then we have $B(u,\bsk) \le B(v,\bsk')$.
\end{enumerate}
\end{lemma}

\begin{proof}
Let us consider the first item. Since $\bsk\geq \bsk'$ is assumed, it follows that
$\rank (C_{u,\bsk'}) \le \rank (C_{u,\bsk})$ and $\rank (C_{u,\bsk' + \bsone_u}) \le \rank (C_{u,\bsk + \bsone_u})$. Thus we have
\[
\frac{B(u,\bsk)}{B(u,\bsk')}
=\left(\frac{b-1}{b}\right)^{\rank (C_{u,\bsk}) - \rank (C_{u,\bsk'})} \left(\frac{1}{b-1}\right)^{\rank (C_{u,\bsk+\bsone_u}) - \rank (C_{u,\bsk'+\bsone_u})}
\le 1,
\]
which proves the statement.

Let us move on to the second item. It follows from the given form of $\bsk$ that
$\rank (C_{v,\bsk'}) = \rank (C_{u,\bsk})$ and $\rank (C_{v,\bsk' + \bsone_v}) \le \rank (C_{u,\bsk + \bsone_u})$, which leads to
\[
\frac{B(u,\bsk)}{B(v,\bsk')}
=\left(\frac{b-1}{b}\right)^{\rank (C_{u,\bsk}) - \rank (C_{v,\bsk'})} \left(\frac{1}{b-1}\right)^{\rank (C_{u,\bsk+\bsone_u}) - \rank (C_{v,\bsk'+\bsone_v})}
\le 1.
\]
Thus we are done.
\end{proof}

As we already know that $\Gamma_{u,\bsk}\leq B(u,\bsk)$ for any $u$ and $\bsk$, it obviously holds that
\[ \Gamma_u\leq \max_{\bsk\in \NN_0^{|u|}}B(u,\bsk)\quad \text{and}\quad \Gamma_u^*\leq \max_{\emptyset \neq v \subseteq u} \max_{\bsk\in \NN_0^{|v|}} B(v,\bsk),\]
for any non-empty $u$. In what follows, we give the maximums over reduced sets, one of which bounds $\Gamma_u$ from above and the other of which is exactly equal to $\Gamma_u^*$.
Define
\begin{equation*}
E_1(u) := \left\{
\bsk \in \NN_0^{|u|} \;\middle|\;
\begin{aligned}
&\text{$C_{u,\bsk+\bsone_u}$ is not full row rank,}\\
&\text{$C_{u,\bsk'+\bsone_u}$ is full row rank for any $\bsk' \le \bsk$ with $\bsk' \neq \bsk$}
\end{aligned}
\right\},\\
\end{equation*}
and
\begin{equation}\label{eq:extremal_condition}
E_2(u) := \left\{
\bsk \in \NN_0^{|u|} \;\middle|\;
\begin{aligned}
&\text{$C_{u,\bsk+\bsone_u}$ is not full row rank,}\\
&\text{$C_{u,\bsk+\bsone_v}$ is full row rank for any $v \subsetneq u$}
\end{aligned}
\right\},
\end{equation}
for $u\subseteq 1{:}s$.
Note that we have $E_2(u) \subseteq E_1(u)$.

The following lemma computes $B(u,\bsk)$ for the cases $\bsk\in E_1(u)$ and $\bsk\in E_2(u)$, respectively.

\begin{lemma}\label{lem:B_for_E12}
For any non-empty $u \subseteq 1{:}s$, the following holds.
\begin{enumerate}
    \item For any $\bszero_u \neq \bsk \in E_1(u)$, we have
    \begin{equation*}
    B(u,\bsk) = \frac{b^{m-|\bsk|}}{(b-1)^{|u|-1}}.
    \end{equation*}
    \item For any $\bsk \in E_2(u)$, we have
    \begin{equation*}
    \Gamma_{u,\bsk}=B(u,\bsk) = \frac{b^{m-|\bsk|}}{(b-1)^{|u|-1}}.
    \end{equation*}
\end{enumerate}
\end{lemma}
\begin{proof}
For the first item, for any $\bszero_u \neq \bsk \in E_1(u)$, there exists $\bsk' \le \bsk$ such that $|\bsk'| = |\bsk|-1$ and $C_{u,\bsk'+\bsone_u}$ is full row rank, and we have $\rank (C_{u,\bsk+\bsone_u}) = \rank (C_{u,\bsk'+\bsone_u}) = |\bsk|+|u|-1$.
This way we obtain
\begin{equation}\label{eq:B_for_E12}
B(u,\bsk)
= b^{m - \rank (C_{u,\bsk})} (b-1)^{\rank (C_{u,\bsk}) - \rank (C_{u,\bsk + \bsone_u})}
= \frac{b^{m - |\bsk|}}{(b-1)^{|u|-1}}.
\end{equation}
Thus we have proved the first item.

For the second item, it follows from the definition of $E_2(u)$ given in \eqref{eq:extremal_condition} that 
$\rank (C_{u,\bsk + \bsone_u}) = |\bsk|+|u|-1$
and
$\rank (C_{u,\bsk + \bsone_v}) = |\bsk|+|v|$
for any $v \subsetneq u$.
Thus \eqref{eq:B_for_E12} is still valid in this case,
and we get the second equality.
Furthermore, by using \eqref{eq:V_rank}, we have
\[
\dim (V(\bsk + \bsone_v) \cap P^\perp)
= |\bsk + \bsone_v| - \rank (C_{\bsk + \bsone_v})
= \chi(v=u),
\]
for any $v \subseteq u$. The inclusion-exclusion principle \eqref{eq:in_ex} leads to
\[
|A(u,\bsk+\bsone_u) \cap P^\perp| = 
\sum_{v \subseteq u} (-1)^{|u|-|v|} |V(u,\bsk+\bsone_v) \cap P^\perp|
= b + \sum_{v \subsetneq u} (-1)^{|u|-|v|} 
= b-1.
\]
By using Lemma~\ref{lem:gain_via_counting}, we obtain
\[ \Gamma_{u,\bsk}= \frac{b^{m-|\bsk|}}{(b-1)^{|u|-1}}. \]
This proves our claim.
\end{proof}

We now evaluate the maximal gain coefficients.

\begin{lemma}\label{lem:gain_bound}
For any non-empty subset $u\subseteq 1{:}s$, we have
\begin{align*}
\Gamma_u &\le \max_{\bsk \in E_1(u)} B(u,\bsk),\\
\Gamma_u^* & = \max_{\emptyset \neq v\subseteq u} \max_{\bsk \in E_1(v)} B(v,\bsk)
= \max_{\emptyset \neq v\subseteq u} \max_{\bsk \in E_2(v)} B(v,\bsk).
\end{align*}

\end{lemma}

\begin{proof}
First let us consider bounding $\Gamma_u$. As mentioned in Remark~\ref{rem:full_rank}, we have $\Gamma_{u,\bsk}=0$ if $C_{u, \bsk+\bsone_u}$ is full row rank, so that is suffices to prove 
\[ \Gamma_{u,\bsk}\leq \max_{\bsk \in E_1(u)} B(u,\bsk), \]
for all $\bsk\in \NN_0^{|u|}$ such that $C_{u, \bsk+\bsone_u}$ is not full row rank. For such $\bsk$, take any $\bsk^*$ which satisfies
\[
|\bsk^*| = \min \{|\bsk'| \mid \text{$\bsk' \le \bsk$ and $C_{u, \bsk'+\bsone_u}$ is not full row rank}\}.
\]
Because of the minimality of $\bsk^*$, it holds that $\bsk^* \in E_1(u)$.
Then it follows from the first item of Lemma~\ref{lem:reduce} that
\begin{align*}
    \Gamma_{u,\bsk} \leq B(u,\bsk)\leq B(u,\bsk^*)\leq \max_{\bsk \in E_1(u)} B(u,\bsk),
\end{align*} 
which proves our claim.

Let us move on to computing $\Gamma_u^*$.
From the above argument on $\Gamma_u$ we have
\[
\Gamma_u^* = \max_{\emptyset \neq v \subseteq u} \Gamma_v
\le \max_{\emptyset \neq v \subseteq u} \max_{\bsk \in E_1(v)} B(v,\bsk)
\]
and from the second item of Lemma~\ref{lem:B_for_E12} we have
\[
\max_{\emptyset \neq v\subseteq u} \max_{\bsk \in E_2(v)} B(v,\bsk)
= \max_{\emptyset \neq v\subseteq u} \max_{\bsk \in E_2(v)} \Gamma(v,\bsk)
\le \max_{\emptyset \neq v\subseteq u} \max_{\bsk \in \NN_0^{|v|}} \Gamma(v,\bsk)
= \Gamma_u^*.
\]
Hence it suffices to prove
\[ B(v,\bsk)\leq \max_{\emptyset \neq v\subseteq u} \max_{\bsk \in E_2(v)} B(v,\bsk), \]
for all $\emptyset \neq v\subseteq u$ and $\bsk\in E_1(v)$. Given such $\emptyset \neq v\subseteq u$ and $\bsk\in E_1(v)$, let us write $w := \{j \in v \mid k_j > 0\}$ and put $\bsk' := (k_j)_{j \in w}$. Take any $z \subseteq v\setminus w$ which satisfies
\[
|z| = \min \{|z| \mid \text{$z \subseteq v\setminus w$ and $C_{v, \bsk + \bsone_{w\cup z}}$ is not full row rank}\},
\]
and let $\bsk^* := (\bsk', \bszero_z) \in \NN_0^{|w \cup z|}$ so that $C_{v, (\bsk^*,\bszero_{v\setminus (w\cup z)}) + \bsone_{w\cup z}}$ is not full row rank.

We now show that $\bsk^* \in E_2(w \cup z)$.
It suffices to show that 
(i) $C_{w \cup z, \bsk^{*} + \bsone_{w \cup \zeta}}$ is full row rank for any $\zeta \subsetneq z$, and that
(ii) $C_{w \cup z, \bsk^{*} + \bsone_{\alpha \cup z}}$ is full row rank for any $\alpha \subsetneq w$.
(i) follows from the minimality of $z$.
We now show (ii). As $\bsk$ is assumed to belong to $E_1(v)$,
$C_{v, \bsk^{*} - \bsone_{w \setminus \alpha} + \bsone_v} = C_{v, \bsk^{*} + \bsone_{\alpha \cup z \cup (v \setminus (z \cup w))}}$ is full row rank.
Thus (ii) holds.
As we now see that $\bsk^* \in E_2(w \cup z)$, it follows from the second item of Lemma~\ref{lem:reduce} that
\begin{align*}
    B(v,\bsk)\leq B(w\cup z,\bsk^*)\leq \max_{\bsk^* \in E_2(w\cup z)} B(w\cup z,\bsk^*)\leq \max_{\emptyset \neq v\subseteq u} \max_{\bsk \in E_2(v)} B(v,\bsk),
\end{align*} 
which proves our claim.
\end{proof}

As in \cite[Section~5]{PO21}, we define
\[ t^*_u := m+1 - \min_{\bsk \in \NN^{|u|}} \left\{|\bsk| \mid \text{$C_{u,\bsk}$ is not full row rank}\right\}
= m+1-|u|-\min_{\bsk \in E_1(u)} |\bsk|.\]
Using Lemmas~\ref{lem:B_for_E12} and \ref{lem:gain_bound}, as a generalization of \cite[Theorem~3 and Corollary~4]{PO21} to prime power base, we arrive at the following result on the maximal gain coefficients.
\begin{theorem}\label{thm:maximal_gain_bound}
For any non-empty $u \subseteq 1{:}s$, where $C_{u,\bsone}$ is full row rank, we have
\[ \Gamma_u \le \frac{b^{t^*_u+|u|-1}}{(b-1)^{|u|-1}}\quad \text{and}\quad \Gamma_u^* = \max_{\emptyset \neq v\subseteq u} \frac{b^{t^*_v+|v|-1}}{(b-1)^{|v|-1}}.\]
In particular, if $C_{1{:}s,\bsone}$ is full row rank, we have
\[ \Gamma = \max_{\emptyset \neq u\subseteq 1{:}s}\frac{b^{t^*_u+|u|-1}}{(b-1)^{|u|-1}}.\]
\end{theorem}

\begin{proof}
As $C_{u,\bsone_u}$ is assumed full row rank, it holds that $\bszero_u \notin E_2(u)$.
Hence, by applying Lemmas~\ref{lem:B_for_E12} to Lemma~\ref{lem:gain_bound} and considering the above definition of $t_u^*$, we have the desired results on $\Gamma_u$ and $\Gamma_u^*$, respectively. The result on $\Gamma$ is a special case of the result on $\Gamma_u^*$ with $u=1{:}s$.
\end{proof}

\begin{remark}\label{rem:maximal_gain_b2}
If $b=2$, we have
\[ \Gamma_u^*= \max_{\emptyset \neq v\subseteq u} 2^{t^*_v+|v|-1}=2^{t_u^*+|u|-1},\]
where the second equality follows from $t_v^*\leq t_u^*+|u|-|v|$, i.e., the maximum over $v\subseteq u$ is attained when $v=u$. Similarly, if $C_{1{:}s,\bsone}$ is full row rank, we have
\[ \Gamma=2^{t^*_{1{:}s}+s-1}.\]
This way we can recover the results shown in \cite[Section~4]{PO21}.
\end{remark}

\begin{example}\label{example}
To illustrate that the results in Theorem~\ref{thm:maximal_gain_bound} are hard to simplify further when $b$ is a prime power other than $2$, let us give an example. Consider a three-dimensional digital net in prime power base $b$ with generating matrices in $\Fb^{3\times 3}$:
\[
C_1 =
\begin{pmatrix}
1 & 0 & 0\\
0 & 1 & 1\\
0 & 0 & 0\\
\end{pmatrix},
C_2 =
\begin{pmatrix}
0 & 1 & 0\\
1 & 0 & 1\\
0 & 0 & 0\\
\end{pmatrix},
C_3 =
\begin{pmatrix}
0 & 0 & 1\\
1 & 1 & 0\\
0 & 0 & 0\\
\end{pmatrix}.
\]
Now let $u_1 = \{1,2,3\}$, $u_2 = \{1\}$,
$\bsk_1 = (1,0,0)$ and $\bsk_2 = (2)$.
Since $\bsk_1\in E_2(u_1)$ and $\bsk_2\in E_2(u_2)$, it follows from the second item of Lemma~\ref{lem:B_for_E12} that
\[
\Gamma_{u_1,\bsk_1} = \frac{b^2}{(b-1)^2} \quad \text{and}\quad 
\Gamma_{u_2,\bsk_2} = b,
\]
for any $b$.
In case of $b=2$, since $C_{1:3,\bsone}$ is full row rank, it follows from either \cite[Corollary~4]{PO21} or our Remark~\ref{rem:maximal_gain_b2} that
the maximum gain coefficient is
$\Gamma = 2^2$, which is attained by $\Gamma_{u_1,\bsk_1}$.
However, if $b \ge 3$, then we have $b^2/(b-1)^2 \le b$ and thus $\Gamma_{u_1,\bsk_1}$ does not coincide with the maximal gain coefficient $\Gamma$, which implies that we generally have
\[ \Gamma\neq \frac{b^{t^*_{1{:}s}+s-1}}{(b-1)^{s-1}},\]
unless $b=2$.

One trivial exceptional case is when $P$ is a digital $(0,m,s)$-net in base $b$ with $m\geq s$. As we have $t_u^*=0$ for any non-empty $u\subseteq 1{:}s$ and $C_{1{:}s,\bsone}$ is full row rank, it follows from Theorem~\ref{thm:maximal_gain_bound} that
\[ \Gamma = \max_{\emptyset \neq u\subseteq 1{:}s}\left(\frac{b}{b-1}\right)^{|u|-1}=\left(\frac{b}{b-1}\right)^{s-1},\]
where the maximum over $u$ is attained uniquely by $u=1{:}s$. This fact was implied already in the proof of \cite[Theorem~3]{Ow97}.
\end{example}

\section*{Acknowledgments}
The authors would like to thank Julian Hofstadler and Art Owen for their useful comments. The authors are also grateful to the reviewers for their helpful comments and suggestions.

\section*{Funding}
The work of T.G. was supported by JSPS KAKENHI Grant Number 20K03744.
The work of K.S. was supported by JSPS KAKENHI Grant Number 20K14326.

\bibliographystyle{plain}
\bibliography{ref.bib}

\end{document}